\documentclass[a4paper,12pt]{amsart}

\usepackage{amsmath,amsthm,amscd,mathrsfs }

\usepackage{amssymb}
\usepackage[all,cmtip]{xy}
\makeatletter
\newcommand\xleftrightarrow[2][]{\ext@arrow 0099{\longleftrightarrowfill@}{#1}{#2}}
\def\longleftrightarrowfill@{\arrowfill@\leftarrow\relbar\rightarrow}
\makeatother

\newtheorem{Theorem}{Theorem}[section]
\newtheorem{Lemma}[Theorem]{Lemma}
\newtheorem{Corollary}[Theorem]{Corollary}
\newtheorem{Proposition}[Theorem]{Proposition}
\newtheorem{Remark}[Theorem]{Remark}
\newtheorem{Notation}[Theorem]{Notation}

\newtheorem{Definition}[Theorem]{Definition}

\newtheorem{Fact}[Theorem]{Fact}

\newcommand{\Hom}{\mathrm{Hom}}

   \DeclareMathOperator{\Assh}{Assh}

\DeclareMathOperator{\depth}{depth}  
 \DeclareMathOperator{\Ker}{Ker} 
\DeclareMathOperator{\Image}{Im} \DeclareMathOperator{\Ass}{Ass} \DeclareMathOperator{\Supp}{Supp} 
   \DeclareMathOperator{\Spec}{Spec}
\DeclareMathOperator{\Gpd}{Gpd}       \DeclareMathOperator{\Gid}{Gid}
\DeclareMathOperator{\amp}{amp}  \DeclareMathOperator{\width}{width}   \DeclareMathOperator{\Ho}{H}   

 \DeclareMathOperator{\id}{id}  
 
 \DeclareMathOperator{\Co}{C}     \DeclareMathOperator{\bas}{B}
\DeclareMathOperator{\Aus}{A}   \DeclareMathOperator{\G}{G}      
  \DeclareMathOperator{\cmd}{cmd}     
        
  \DeclareMathOperator{\Max}{Max}  
\DeclareMathOperator{\E}{E}

\usepackage{amsmath, amsthm}
\usepackage{chemarr}

\title
{G-Gorenstein complexes}

\author{Maryam Akhavin}
\author{Eero Hyry}

\address{
Mathematics and Statistics\\
School of Information Sciences\\ 
University of Tampere\\
FIN-33014 Tampereen yliopisto\\ 
Finland}

\email{maryam.akhavin@uta.fi}
\email{eero.hyry@uta.fi}

\begin{document}

\begin{abstract}
We present in the context of Gorenstein homological
algebra the notion of a ``G-Gorenstein complex'' as the counterpart of
the classical notion of a Gorenstein complex. In particular, we investigate equivalences between 
the category of G-Gorenstein complexes of fixed dimension and the $G$-class of modules.  

\end{abstract}

\maketitle

\section{Introduction}

Gorenstein homological algebra is the relative version of homological algebra, where classical injective and projective modules are replaced by Gorenstein injective and Gorenstein projective modules, respectively. The study of Gorenstein homological algebra goes back to Auslander and Bridger. They introduced the notion of a Gorenstein dimension of a finitely generated module over a commutative Noetherian ring 
(see~\cite{AB}). Gorenstein dimension characterizes Gorenstein rings like projective dimension does for regular rings. In order to extend this theory to arbitrary modules, Enochs and Jenda defined the notions of a Gorenstein projective and Gorenstein injective module (see~\cite{enjemz}). 
 
Gorenstein complexes, defined by Grothendieck in~\cite{hart}, play a crucial role in his theory of duality. Sharp initiated in~\cite{sharp} the study of Gorenstein modules from the point of view of commutative algebra. Following the maxim that every result in classical homological algebra has a counterpart in Gorenstein homological algebra, as suggested by Holm in~\cite{hthesis}, the purpose of this article is to introduce an analogue of the notion of a Gorenstein complex in the context of Gorenstein homological algebra. We can extend several properties of Gorenstein modules proved by Sharp to the case of $G$-Gorenstein complexes. Our work generalizes that of Aghajani and Zakeri who introduced in~\cite{zagh} the notion of a G-Gorenstein module (see also~\cite{khashayar}).

Let $R$ be a commutative Noetherian ring. The derived category of bounded complexes of $R$-modules with finitely generated homology is denoted by $D^f_b(R)$. Generalizing the definition of a Gorenstein complex given in~\cite{hart} we define a complex $M\in D^f_b(R)$ to be G-Gorenstein if it is Cohen-Macaulay and the local cohomology modules $\Ho^i_{pR_p}(M_p)$ are Gorenstein injective for all $i\in \mathbb Z$ and $p\in\Spec R$. 

From now on we assume that $(R,m)$ is a local ring admitting a dualizing complex. It  comes out in Proposition~\ref{d=gi} that the $G$-Gorensteiness of $M$ is equivalent to $\dim_R M =\depth_R M=\Gid_R M$. This is further equivalent to $M$ being of finite Gorenstein injective dimension and having $\depth_RM=\depth R-\inf M$. Recall the open question concerning the analogue of Bass's theorem in Gorenstein homological algebra: Does the existence of an $R$-module of finite Gorenstein injective dimension imply that $R$ is Cohen-Macaulay (see~\cite[Question 3.26]{CFH})? Regarding this question we point out in Corollary~\ref{s2cm} that if $R$ satisfies Serre's condition $S_2$, then the existence of a $G$-Gorenstein module always implies that $R$ is Cohen-Macaulay. 

If $M\in D^f_b(R)$ is a complex of finite Gorenstein injective dimension, then the biduality morphism $L\longrightarrow {\bf R}\Hom_R({\bf R}\Hom_R(L,M),M)$ cannot be an isomorphism for $L\in D^f_b(R)$ unless 
$M$ is a dualizing complex. This was observed by Christensen in~\cite[Proposition 8.4]{Auscat}. Nevertheless, it turns out that if $M$ is $G$-Gorenstein, then biduality preserves depth. In fact, we prove
in our first main result Theorem~\ref{2} that among complexes of finite Gorenstein injective dimension $G$-Gorenstein complexes are characterized by the equality $$\depth {\bf R}\Hom_R({\bf R}\Hom_R(L,M),M)= \depth L$$ for all complexes $L\in D^f_b(R)$ of finite projective or injective dimension.

Let $M\in D^f_b(R)$. Our Theorem~\ref{asly1} and Theorem~\ref{asly2} show that the following conditions are equivalent:
\begin {itemize}
\item[(1)] $M$ is a G-Gorenstein complex of dimension $t$;
\item[(2)] $M\simeq \Hom_R(K,\Sigma^{-t}D_R)$ for some $K\in G(R);$
\item[(3)]$M\simeq\Sigma^{-t}D_R\otimes_RN $ for some $N\in G(R);$ 
\item[(4)] ${\bf R}\Hom_R( \Sigma^{-t}D_R,M)\simeq N$ for some $N\in G(R).$
\end{itemize}
Here $D_R$ denotes the dualizing complex normalized with $\sup D_R=\dim R$ and $G(R)$ is the G-class of modules. As usual, the symbol ``$\simeq$'' indicates an isomorphism in $D(R)$.

Let $D_{t-GGor}(R)$ denote the full subcategory of $D^f_b(R)$ of $G$-Gorenstein complexes of dimension $t$. In more abstract terms, we can then say that there is a diagram
\[
\xymatrixcolsep{10pc}
\xymatrixrowsep{4pc}
\xymatrix{ 
D_{t-GGor}(R)
\ar@<1ex>[r]^{H_t({\bf R}\Hom_R(-,D_R))}
\ar[d]_{}
& \G(R)^{\mathrm{opp}}
\ar@<1ex>[l]^{\Sigma^{-t}{\bf R}\Hom_R(-,D_R)} 
\ar[d]^{\Hom_R(-,R)}
\\
D_{t-GGor}(R)
\ar@<1ex>[u]^{\id}
\ar@<1ex>[r]^{H_{-t}({\bf R}\Hom_R(D_R,-))}
& \G(R)
\ar@<1ex>[l]^{\Sigma^{-t}D_R\otimes ^L_R-} 
\ar@<1ex>[u]^{}
} 
\]
of equivalences of categories, where the horizontal arrows are quasi-inverses of each other. The diagram is commutative up to canonical isomorphisms. The upper equivalence is the restriction of an equivalence 
between the full subcategory of $D^f_b(R)$ of Cohen-Macaulay complexes of dimension $t$ and the category of finitely generated $R$-modules. The latter equivalence was first observed by Yekutieli and 
Zhang in~\cite{YZ} and later utilized by Lipman, Nayak and Sastry in~\cite{lnsfunctorial}. The lower equivalence comes from Foxby equivalence
$$\xymatrixcolsep{8pc}\xymatrix{ 
 \Aus(R)
\ar@<1ex>[r]^{D_R\otimes ^L_R-}
& \bas(R) 
\ar@<1ex>[l]^{{\bf R}\Hom_R(D_R,-)} 
}
$$
between the Auslander and the Bass classes. 

Inspired by the theory of Gorenstein objects in triangulated categories developed by Asadollahi and Salarian in~\cite{asad}, we want to consider $G$-Gorenstein complexes as Gorenstein objects. 
Let $t\in \mathbb Z$. Set $D=\Sigma^{-t}D_R$. We look at towers  
$$
\xymatrix{
\cdots\ar[r]&D^{\oplus n_{i+1}}\ar[rd]^{g_{i+1}}\ar[rr]^{d_{i+1}}&&D^{\oplus n_i}\ar[rd]^{g_i}\ar[rr]^{d_i}&&D^{\oplus n_{i-1}}\ar[rd]\ar[r]&\cdots\\
M_{i+1}\ar[ru]&&M_i\ar[ru]^{f_i}\ar@{.>}[ll]&&M_{i-1}\ar@{.>}[ll]\ar[ru]^{f_{i-1}}&&M_{i-2}\ar@{.>}[ll]} 
$$
of exact triangles in $D^f_b(R)$, where $d_i=f_{i-1}g_i$. It then comes out in Theorem~\ref{dcomplete} that a complex $M\in D^f_b(R)$ is a $G$-Gorenstein complex of 
dimension $t$ if and only if $M\simeq M_i$ for some $i$ in a tower of triangles, where the triangles are both $\Hom_{D(R)}(D,-)$-exact and $\Hom_{D(R)}(-,D)$-exact (see 
Definition~\ref{exacttriangle}). In Corollary~\ref{K_R-exact} we look at the special case where
$R$ is Cohen-Macaulay with the canonical module $K_R$. Then a finitely generated $R$-module $M$ is G-Gorenstein if and only if $M$ appears as a kernel in an exact complex 
of $R$-modules $$ \cdots\rightarrow K_R^{\oplus n_{i+1}}\stackrel{d_{i+1}}{\rightarrow}K_R^{\oplus n_i} \stackrel{d_i}{\rightarrow}K_R^{\oplus n_{i-1}}\rightarrow\cdots $$
which is both $\Hom_R(K_R,-)$- and $\Hom_R(-,K_R)$-exact. This means that $G$-Goren\-stein modules are exactly the $K_R$-Gorenstein projective 
modules in the sense of~\cite{enje2004}.

We now describe the contents of this paper. In Section 2 we recall some facts of hyperhomological algebra needed in the sequel. In Section 3 we recall some basic properties of Cohen-Macaulay
complexes. In particular, extending the definition Schenzel gave for modules in~\cite{schenzellec}, we introduce the notion of a module of deficiency of a complex, which is the main 
technical tool of this article. We start the investigation of G-Gorenstein complexes in Section 4. In Section 5 we study their behaviour in the equivalences of categories mentioned above. 
In Section 6, we show that $G$-Gorenstein complexes can be considered as Gorenstein objects with respect to a suitable subcategory of $D(R)$. For notation and terminology, see the section Preliminaries below.

\section{Preliminaries}

The purpose of this section is to fix notation and recall some definitions and results of hyperhomological algebra relevant to this article. As a general reference, we mention~\cite{C.L.N} and references therein. For more details, see also~\cite{f.vanishloho} and~\cite{hart}. 

In the following $R$ is always a commutative Noetherian ring. If $R$ is local, then $m$ denotes the maximal ideal and
$k$ the residue field of $R$.

Throughout this article we work within the derived category $D(R)$ of $R$-modules. We use homological grading so that the objects of $D(R)$ are complexes of $R$-modules
of the form $$M: \quad\ldots\stackrel{d_{i-1}}{\rightarrow} M_{i+1}\stackrel{d_{i+1}}{\rightarrow} M_i\stackrel{d_i}{\rightarrow} M_{i-1} \ldots \quad .$$
The derived category is triangulated, the suspension functor $\Sigma$ being defined by the formulas $(\Sigma M)_n= M_{n-1}$ and $d^{\Sigma M}_n=-d_n$. 
The symbol ``$\simeq$'' is reserved for isomorphisms in $D(R)$. We use the subscript ``$b$'' to denote the homological boundness and the superscript ``$f$'' to denote the homological 
finiteness. So the full subcategory of $D(R)$ consisting of complexes with finitely generated homology modules is denoted by $D^f(R)$. As usual, we identify the 
category of $R$-modules as the full subcategory of $D(R)$ of complexes $M$ satisfying $H_i(M)=0$ for $i\neq 0$. For a complex
$M\in D(R)$, by $\sup M$ and $\inf M$, we mean its homological supremum and infimum. The \textit{amplitude} $\amp M = \sup M - \inf M$. We use the standard notations $\otimes_R^L$ and ${\bf R}\Hom$ for the derived tensor product and the derived $\Hom$ functor.

The \textit{support} of a complex $M\in D(R)$ is the set $$\Supp_RM=\left\{p\in \Spec R\mid M_p\not\simeq 0 \right\}.$$ The \textit{Krull dimension} $$ \dim_RM  =\sup \left\{\dim R/p -\inf M_p 
\mid p\in {\Supp_RM}\right\}. $$ When $(R,m)$ is local, the \textit{width} and \textit{depth} of $M$ are defined by the formulas $\width_R M=\inf(k\otimes_L M)$ and $\depth_RM =-\sup {\bf R}\Hom_R(k,M)$,
respectively. 

If $(R,m)$ is a local ring, the \textit{derived local cohomology} functor with respect to  $m$ is denoted by ${\bf R}\Gamma_m$. As usual, we set $H_m^i(-)=H_{-i}({\bf R}\Gamma_m(-))$ for all $i\in \mathbb Z$. Note that
\begin{equation}\label{lidagger}-\inf {\bf R}\Gamma_m(M)=\dim_RM\end{equation}
and
\begin{equation}\label{lsdagger}-\sup{\bf R}\Gamma_m(M)=\depth_RM \end{equation}
If $R$ admits a dualizing complex, we denote by $D_R$ the dualizing complex normalized with $\sup D_R=\dim R$ and $\inf D_R=\depth R$. The \textit{dagger dual} of a complex $M\in D^f_b(R)$ is $M^\dagger = {\bf R}\Hom_R(M,D_R)$. We obtain a contravariant functor $(-)^\dagger\colon D^f_b(R)\rightarrow D^f_b(R)$.  The canonical morphism $M \rightarrow M^{\dagger\dagger}$ induces the biduality $M\simeq M^{\dagger\dagger}$, which is called the \textit{dagger duality} for $M$. The \textit{local duality} says that 
\begin{equation}\label{LDT}{\bf R}\Gamma_m(M)\simeq\Hom_R(M^\dagger,E_R(k)),\end{equation} where $E_R(k)$ denotes the injective envelope of $k$. We will frequently use the formulas
\begin{equation}\label{sdagger}\sup M^{\dagger}=\dim_RM\end{equation}
and
\begin{equation}\label{idagger}\inf M^{\dagger}=\depth M.\end{equation} 
Also observe that \begin{equation}\label{daggerloc}(M_p)^\dagger\simeq  \textstyle  \Sigma^{-\dim R\slash p}(M^\dagger)_p \end{equation} for all $p\in \Spec R$. 
Here the dagger dual  on the left-hand side is taken with respect to the normalized dualizing complex of the localization $R_p$.

Let $R$ be a ring. Recall that an $R$-module $N$ is called \textit{Gorenstein injective}, if there is an exact complex $I$ of injective $R$-modules such that the complex 
$\Hom_R(J,I)$ is exact for every injective $R$-module $J$, and that $N$ appears as a kernel in $I$. For $M\in D_b(R),$ the Gorenstein injective dimension of $M$, denoted by $\Gid_RM$, is defined as the infimum of all integers $n$ such that there exists a complex $I$ of Gorenstein injective $R$-modules for which $I\simeq M$ in $D(R),$ and $I_i=0$ if 
$i>-n$. Note that $\Gid_R\textstyle  \Sigma^sM=-s+\Gid_RM$. The notions of a \textit{Gorenstein projective} module and a \textit{Gorenstein flat} module are defined similarly.  The 
\textit{G-class} of modules, denoted by $G(R)$, consists of all finitely generated Gorenstein projective, or, equivalently,  Gorenstein flat $R$-modules.

Let $(R,m)$ be a local ring admitting a dualizing complex $D$. The \textit{Auslander class} $\Aus(R)$ and the \textit{Bass class} $\bas(R)$ with respect to $D$ are full subcategories of $D_b(R)$ such that
the functors $D\otimes ^L_R-$ and ${\bf R}\Hom_R(D,-)$ restrict to quasi-inverse equivalences between them. This \textit{Foxby equivalence} induces even an equivalence between their restrictions $\Aus^f(R)$ and $\bas^f(R)$ to the category $D^f_b(R)$. It is an important fact that $\Aus(R)$ and $B(R)$ consist exactly of all bounded complexes of finite Gorenstein projective dimension and of finite Gorenstein injective dimension, respectively (see~\cite[Theorem 4.1 and Theorem 4.4]{c.h.f.functurial}).

\section{Cohen-Macaulay complexes}

This section is partly of preliminary nature. We record here for the convenience of the reader some facts about Cohen-Macaulay complexes, which will be used
in the rest of this article. 

Let $(R,m)$ be a local ring. The \textit{Cohen-Macaulay defect} of a complex $M\in D_b(R)$ is the number
$$\cmd_R M=\dim_RM-\depth_R M.$$ It is known that if $M\in D^f_b(R)$ and $M\not\simeq0$, then $\cmd_RM\ge 0$. 
If $\cmd M=0$, then $M$ is called \textit{Cohen-Macaulay}. This is equivalent to complex $M_p$ 
being Cohen-Macaulay for every $p\in\Supp_RM$. Moreover, we then have
\begin{equation}\label{ineq}\dim_RM=\dim_{R_p}M_p+\dim R\slash p.\end{equation}
When $R$ is a non-local ring, a complex $M\in D^f_b(R)$ is defined to be Cohen-Macaulay if the complex $M_m$ 
is Cohen-Macaulay for all $m\in\Max(R)\cap \Supp_RM$.

If $R$ is a ring and $N$ is an $R$-module we use the notation 
$$\Assh_R N=\{p\in \Supp_R N\mid \dim(R\slash p)=\dim N\}.$$

\begin{Proposition}\label{hsk}Let $(R,m)$ be a local ring and let $M\in D^f_b(R)$ be a Cohen-Macaulay complex.  Then 
$$\Ass_R\Ho_s(M)=\left\{p\in\Supp_RM\mid \dim R\slash p=\dim_RM+s\right\},$$
where $s=\sup M$. In particular, $\Ass_R\Ho_s(M)=\Assh_R\Ho_s(M)$. 
\end{Proposition}

\begin{proof} Let $p\in \Supp_RM$. By~\cite[(A.6.1.2)]{C.L.N} we know that
$p\in\Ass_R\Ho_s(M)$ if and only if $\depth_{R_p}M_p=-s$. Since $M$ is Cohen-Macaulay, we have $\depth_{R_p}M_p=\dim_{R_p}M_p$. 
It then follows from formula~$(\ref{ineq})$ that $p\in\Ass_R\Ho_s(M)$ if and only if $\dim_RM=\dim R\slash p-s$. 

\end{proof}

As in the case of modules, one can characterize Cohen-Macaulay complexes in terms of vanishing local cohomology:

\begin{Proposition}\label{delta and cmd} Let $R$ be a ring. If $M\in D^f_b(R)$, then the following statements are equivalent:
\begin{itemize}
\item[a)]$M$ is Cohen-Macaulay;
\item[b)]$\Ho^i_{pR_p}(M_p)=0$ for all $p\in\Spec R$ and $i\neq\dim_{R_p}M_p$;
\item[c)]$\Ho^i_{mR_m}(M_m)=0$ for all $m\in\Max R$ and $i\neq\dim_{R_m}M_m$.
\end{itemize} 
\end{Proposition}
\begin{proof}Immediate from formulas~$(\ref{lidagger})$ and~$(\ref{lsdagger})$.
\end{proof}

\begin{Remark}\label{hartrmk}Let $R$ be a ring. If $X\subseteq\Spec R$, then a filtration of $X$ is a descending sequence $$\mathcal{F}^{\bullet}\colon\quad \ldots \supseteq F^{i-1}\supseteq F^i\supseteq F^{i+1} \supseteq\ldots $$ of subsets of $X$ such that $\bigcap_i F^i=\emptyset$, $F^i=X$ for some $i\in \mathbb Z$ and each $p\in F^i \setminus  F^{i+1}$ is a minimal element of $F^i$ with respect to inclusion. 
In~\cite[p.~238]{hart} a complex $M\in D^f_b(R)$ is defined to be Cohen-Macaulay with respect to $\mathcal{F}^{\bullet}$ if $\Ho^n_{pR_p}(M_p)=0$ 
for all $n\neq i$ and $p\in F^i \setminus F^{i+1}$.

Given a complex $M\in D^f_b(R)$, set $$F^i=\left\{p\in\Supp_RM\mid \dim_{R_p}M_p\geq i\right\}$$ for all $i\in \mathbb Z$. It is easily checked that this gives a filtration of $\Supp M$ (the so called ``$M$-height-filtration'').  Proposition~\ref{delta and cmd} then implies that $M$ is Cohen-Macaulay in the sense mentioned earlier if and only if $M$ is Cohen-Macaulay with respect to this filtration.  
Let $E(M)$ denote the corresponding Cousin complex. Recall that $E(M)$ is a complex $\ldots \rightarrow E(M)^i \rightarrow E(M)^{i+1}\rightarrow\ldots $ with $$E(M)^i=\bigoplus_{\dim_{R_p}M_p=i} \Ho^i_{pR_p}(M_p).$$ Contrary to our convention, we follow here the general tradition and grade the Cousin complex cohomologically. For more details about Cousin complexes
we refer to~\cite[3.2]{lnsfunctorial}. It now follows from~\cite[Chapter IV, Proposition 3.1]{hart} that $M\simeq E(M)$ if and only if $M$ is Cohen-Macaulay complex. Note that if $M$ is a module, then the Cousin complex studied by Sharp (see~\cite{sharptang}, for example) is the complex $0\rightarrow M\rightarrow E(M)$. 

\end{Remark}

In order to investigate the structure of a Cohen-Macaulay complex, it is useful to introduce the notion of the module of deficiency of a complex.
In the module case this was done by P.~Schenzel in~\cite[p.~60]{schenzellec}.

\begin{Definition}\label{kanonicalforcomplex}
Let $(R,m)$ be a local ring admitting a dualizing complex and let $M\in D^f_b(R)$. For every $i\in \mathbb Z$, set $K^i_M= \Ho_i(M^\dagger)$. The modules $K^i_M$ are called 
the modules of deficiency of the complex $M$. Moreover, we set $K_M=K^{\dim_R M}_M$, and say that $K_M$ is the canonical module of $M$. 
\end{Definition}

\begin{Remark} The modules of deficiency are clearly finitely generated. Using formulas $(\ref{sdagger})$ and $(\ref{idagger})$, we get
$K^i_M=0$ for $i<\depth_RM$ and $i>\dim_RM$. More precisely, by local duality $\Ho^i_m(M)\cong \Hom_R(K_M^i,E_R(k))$
for all $i\in \mathbb Z$.
\end{Remark}

\begin{Lemma}\label{changcanonical}
Let $(R,m)$ be a local ring admitting a dualizing complex and let $M\in D^f_b(R).$ Then 

\begin{itemize}
\item[a)] $(K^i_M)_p\cong K^{i-\dim R\slash p}_{M_p}$ for every $p\in \Supp_RM$;
\item[b)]If $p\in\Supp_RM$ with $\dim_RM=\dim_{R_p}M_p+\dim R\slash p$, then $(K_M)_p\cong K_{M_p}$. In particular, this holds if
$M$ is Cohen-Macaulay.
\end{itemize}
\end{Lemma}
\begin{proof}

\smallskip
\noindent
$a)$ By formula~$(\ref{daggerloc})$ $$(K^i_M)_p\cong \Ho_i((M^\dagger)_p)\cong \Ho_{i-\dim R\slash p}((M_p)^{\dagger})= K^{i-\dim R\slash p}_{M_p}.$$
  
\smallskip
\noindent
 
$b)$ By a) we immediately get $$(K_M)_p=(K_M^{\dim_{R_p}M_p+\dim R\slash p})_p\cong K_{M_p}.$$ The last statement
is then a consequence of formula~$(\ref{ineq})$. 

\end{proof}

\begin{Notation}Let $R$ be a ring. Let $t\in \mathbb Z$. We denote  by $D_{t-CM}(R)$ the full subcategory of $D^f_b(R)$ of Cohen-Macaulay complexes of dimension $t$. 
\end{Notation}

\begin{Proposition}\label{yz} Let $(R,m)$ be a local ring admitting a dualizing complex and let $M\in D^f_b(R)$. Then the following statements are equivalent:
\begin{itemize}
\item[a)]$M$ is Cohen-Macaulay;
\item[b)]$M^\dagger\simeq \Sigma^{\dim_RM}K_M$;
\item[c)]$M^\dagger\simeq \Sigma^{t}N$ for some finitely generated $R$-module $N$ and $t\in \mathbb Z$.
\end{itemize}
It follows that the functors
$$\xymatrixcolsep{8pc}\xymatrix{ 
D_{t-CM}(R) 
\ar@<1ex>[r]^{K_{-}}
& (R\mathrm{-mod})^{\mathrm{opp}}
\ar@<1ex>[l]^{\Sigma^{-t}(-)^\dagger} 
}
$$ are quasi-inverses of each other, and thus provide an equivalence of categories.

\end{Proposition}

\begin{proof}

\smallskip
\noindent$a)\Rightarrow b)$: Because $K^i_M=0$ if $i\neq\dim_RM$, we get $M^{\dagger}\simeq\Sigma^{\dim_RM}K_M$.

\smallskip

\noindent$b)\Rightarrow c)$: This is trivial.

\noindent$c)\Rightarrow a)$: Since now $\sup M^{\dagger}=\inf M^{\dagger}$, we have $\cmd_RM=0$ by formulas~$(\ref{sdagger})$ and $(\ref{idagger})$
implying that $M$ is Cohen-Macaulay.

\end{proof}

The equivalence of categories of Proposition~\ref{yz} is due to Yekutieli and Zhang (see~\cite[Theorem 6.2]{YZ}). In particular, we also recover
the following (see~\cite[Remark 6.3]{YZ}):

\begin{Corollary}\label{abeli}Let $(R,m)$ be a local ring admitting a dualizing complex. Then $D_{t-CM}(R)$ is an abelian subcategory of $D^f_b(R).$
\end{Corollary}

\begin{Corollary}\label{propkkm1}Let $(R,m)$ be a local ring admitting a dualizing complex and let $M\in D^f_b(R)$ be a Cohen-Macaulay complex. Then
\begin{itemize} 
\item[a)]$\dim_RK_M=\dim_RM+\sup M$;
\item[b)]$\depth_RK_M=\dim_RM+\inf M$.
\end{itemize}
In particular, $K_M$ is Cohen-Macaulay if and only if $M$ is a module up to a suspension.
\end{Corollary}
\begin{proof} Since $K_M\simeq \Sigma^{-\dim_RM}M^\dagger$ by Proposition~\ref{yz}, the claim follows	from formulas $(\ref{sdagger})$ and $(\ref{idagger})$.
\end{proof}

\begin{Proposition}\label{supprtcm}Let $(R,m)$ be a local ring admitting a dualizing complex and let $M\in D^f_b(R)$ be a Cohen-Macaulay complex. Then 
\begin{itemize}
\item[a)] $\Supp_RK_M=\Supp_RM$;
\item[b)] $\Assh_RK_M=\Ass_R\Ho_s(M)$.
\end{itemize}
\end{Proposition}
\begin{proof}
$a)$ By Proposition~\ref{yz} $M^{\dagger}\simeq\Sigma^{\dim_RM}K_M$. So $\Supp_RK_M=\Supp_RM^{\dagger}$. On the other hand,
because of formula~$(\ref{daggerloc})$, we have $\Supp_RM^{\dagger}\subseteq\Supp_RM$. Then $\Supp_RM\subseteq\Supp_RM^{\dagger}$  by biduality so that $\Supp_RM=\Supp_RM^{\dagger}=\Supp_RK_M$. 

$b)$ By a) and Proposition~\ref{propkkm1} $\Assh_R K_M$ consists of $p\in \Supp M$ satisfying $\dim R\slash p=\dim_RM+\sup M$. The claim then follows from Proposition~\ref{hsk}.
\end{proof}

\section{Properties of $G$-Gorenstein complexes}

Recall from~\cite[p.~248]{hart} that a complex $M\in D^f_b(R)$ is called a Gorenstein complex if it is Cohen-Macaulay and the local cohomology modules $\Ho^i_{pR_p}(M_p)$ are injective 
$R_p$-modules for all $i\in \mathbb Z$ and $p\in\Spec R$. Motivated by this, we now give 

\begin{Definition}\label{G-gorcom}
Let $R$ be a ring. A complex $M\in D^f_b(R)$ is called a G-Gorenstein complex if it is a Cohen-Macaulay and the local cohomology modules $\Ho^i_{pR_p}(M_p)$ are Gorenstein injective $R_p$-modules for all $i\in \mathbb Z$ 
and $p\in\Spec R$.
\end{Definition}

\begin{Remark}Suppose that $R$ admits a dualizing complex. Then $\Ho^i_{pR_p}(M_p)$ is Gorenstein injective as an $R_p$-module if and only if it is Gorenstein injective as an $R$-module (use~\cite[Lemma 3.2]{zagh} 
and~\cite[Proposition 5.5]{c.h.f.functurial}). Furthermore, we know by~\cite[Theorem 6.9]{c.h.f.functurial} and~\cite[Theorem 2.6]{hghd} that the class of Gorenstein injective $R$-modules is closed under direct sums and summands. Let $E(M)$ denote the Cousin complex of $M$ with respect to the ``$M$-height filtration'' as in Remark~\ref{hartrmk}. The condition $\Ho^i_{pR_p}(M_p)$ is Gorenstein injective for all $i\in \mathbb Z$ and $p\in\Spec R$ is thus equivalent to the components of $E(M)$ being Gorenstein injective. Recalling from Remark~\ref{hartrmk} that $M$ is Cohen-Macaulay if and only if $M\simeq E(M)$, we conclude that $M$ is $G$-Gorenstein if and only if its Cousin complex $E(M)$ provides a Gorenstein injective resolution of $M$. In particular, Definition~\ref{G-gorcom} generalizes the definition of a $G$-Gorenstein module Aghajani and Zakeri gave in~\cite[Definition 3.1]{zagh}. 
\end{Remark}

In the presence of a dualizing complex we could reformulate Definition~\ref{G-gorcom} as follows by using only maximal ideals:

\begin{Proposition}\label{G-Gor for local}Let $R$ be a ring admitting a dualizing complex and let $M\in D^f_b(R)$. Then $M$ is a G-Gorenstein complex if and only if $M$ is Cohen-Macaulay and 
the local cohomology modules $\Ho^i_m(M)$ are Gorenstein injective $R_m$-modules for all $m\in\Max(R)$ and $i\in \mathbb Z$.
\end{Proposition}

\begin{proof} Let $m\in\Max(R)$ and $i\in \mathbb Z$. It is enough to show that if $\Ho^i_{m}(M)$ is Gorenstein injective, then $\Ho^i_{pR_p}(M_p)$ is Gorenstein injective for all $p\in\Spec R$ with
$p\subset m$. 
Since $R$ admits a dualizing complex, it follows from~\cite[Proposition 5.5]{c.h.f.functurial} that $\Ho^i_{mR_m}(M_m) \cong (\Ho^i_m(M))_m$ is Gorenstein injective. We may thus assume that $R$ is local. We have
$\Ho^i_m(M)\cong\Hom_R(K^i_M,\E_R(k))$ The module $\Ho^i_m(M)$ now being Gorenstein injective, this implies by~\cite[Theorem 6.4.2]{C.L.N} that $K^i_M$ is Gorenstein flat. By
Lemma~\ref{changcanonical} a) $K^i_{M_p}\cong(K^{i+\dim R\slash p}_M)_p$. 
So $K^i_{M_p}$ is Gorenstein flat. Using~\cite[Theorem 6.4.2]{C.L.N} again shows that  $\Ho^i_{pR_p}(M_p)\cong\Hom_{R_p}(K^i_{M_p},\E_{R_p}(R_p\slash{pR_p}))$ is Gorenstein injective as wanted.
\end{proof}

In analogy with Sharp's result~\cite[Theorem 3.11 (vi)]{sharp} on Gorenstein modules, we want to characterize $G$-Gorenstein complexes in terms of Gorenstein injective dimension. First we need two lemmas.

\begin{Lemma}\label{gdim}
Let $(R,m)$ be a local ring admitting a dualizing complex and let $M\in D^f_b(R)$. Then $\Gid_RM=\Gid_R{\bf R}\Gamma_m(M)$.
\end{Lemma}
\begin{proof} Since $R$ admits a dualizing complex, we know by~\cite[Theorem 5.9]{c.h.f.functurial} that $\Gid_R{\bf R}\Gamma_m(M)$ and $ \Gid_RM$ are simultaneously finite. So we can suppose that both of
them are finite. We will use~\cite[Theorem 6.8]{c.h.f.functurial} according to which $$\Gid N=\sup\{\depth R_p-\width_{R_p}N_p\mid p\in \Spec R\}$$ for any $N\in D_b(R)$. Here $\width_{R_p}N_p=\infty$
if $p\not\in \Supp_R N$. Noting that $\Supp_R{\bf R}\Gamma_m(M)=\left\{m\right\}$, it then follows that $$\Gid_R{\bf R}\Gamma_m(M)=\depth R-\width _R{\bf R}\Gamma_m(M).$$ Recall from~\cite[Proposition 3.1.2]{lipman2}, for example, that ${\bf R}\Gamma_m(M)\simeq \Co_m(R)\otimes^L_R M$, where $\Co_m(R)$ denotes the \v{C}ech complex on $m$. Because $\width_R C_m(R)=0$, \cite[(A.6.5)]{C.L.N} implies that $\width_R{\bf R}\Gamma_m(M)=\width_RM$. Furthermore, we have $\width _R M=\inf M$, since $M\in D^f_b(R)$. On the other hand, by \cite [Theorem 6.3]{c.h.f.functurial}
$\Gid M=\depth R-\inf M$. We can thus conclude that $\Gid_R{\bf R}\Gamma_m(M)=\Gid_R M$, as wanted.
\end{proof}

\begin{Lemma}\label{G-dim}Let $(R,m)$ be a local ring admitting a dualizing complex. If $M\in D^f_b(R)$ has finite Gorenstein injective dimension, then
$\Gid_RM\geq\dim_RM$.
 \end{Lemma}
\begin{proof} 
One has $\Gid_RM=\Gpd_RM^{\dagger}$ by~\cite[Corollary 6.4]{c.h.f.functurial}. Obviously we have $\Gpd_RM^{\dagger}\geq\sup M^{\dagger}$. So the claim results from formula~$(\ref{sdagger})$.
\end{proof}

We are now ready to prove

\begin{Proposition}\label{d=gi}Let $(R,m)$ be a local ring admitting a dualizing complex, and let $M\in D^{f}_b(R).$ Then the following statements are equivalent:\begin {itemize}
\item[a)] $M$ is a G-Gorenstein complex;
\item[b)] $\dim_R M = \depth_R M = \Gid_RM$;
\item[c)] The Gorenstein injective dimension of $M$ is finite and $$\depth M=\depth R-\inf M.$$ 
\end{itemize}
\end{Proposition}
\begin{proof}
$a) \Leftrightarrow b)$: Set $\dim_RM=t$. In any case, $M$ is Cohen-Macaulay. So ${\bf R}\Gamma_mM\simeq\Sigma^{-t}\Ho^t_m(M)$ by Proposition~\ref{delta and cmd}. By Lemma~\ref{gdim} we then have $$\Gid_RM=\Gid_R\textstyle  \Sigma^{-t}\Ho^t_m(M)=t+\Gid_R\Ho^t_m(M).$$ This shows that $\Ho^t_m(M)$ is Gorenstein injective if and only if $\Gid_RM=t$, as needed. 

$b) \Leftrightarrow c)$: Because $\Gid M$ is finite, we know from~\cite[Theorem 6.3]{c.h.f.functurial} that $\Gid M=\depth R-\inf M$. Since $\dim_R M\ge \depth_R M$, it 
follows from Lemma~\ref{G-dim} that $\depth_R M=\depth R-\inf M$ if and only if $\dim_R M=\depth_R M=\Gid M$. 

\end{proof}

We immediately recover~\cite[Theorem 3.8]{zagh}.

\begin{Corollary}\label{Zakeri} Let $(R,m)$ be a local ring of dimension $d$ admitting a dualizing complex. If $R$ is Cohen-Macaulay, then a finitely generated $R$-module is 
$G$-Gorenstein if and only if it is a maximal Cohen-Macaulay module of finite Gorenstein injective dimension.
\end{Corollary}

We also observe the following:

\begin{Corollary}\label{CMcor} Let $(R,m)$ be a local ring of dimension admitting a dualizing complex. If $R$ admits a G-Gorenstein module with $\dim_RM=\dim R$,  
then $R$ is Cohen-Macaulay.
\end{Corollary}

\begin{Proposition}\label{ass w}Let $(R,m)$ be a local ring admitting a dualizing complex and let $M$ be a G-Gorenstein complex. Then 
$$\left\{p\in\Supp_RM\mid \dim R\slash p-\inf M_p=\dim_R M\right\}=\Ass R\cap\Supp_RM.$$
\end{Proposition}

\begin{proof}Let $p \in \Supp_RM$. Since $M_p$ is $G$-Gorenstein, we now have $$\dim_{R_p}M_p=\depth R_p-\inf M_p$$ 
by Proposition~\ref{d=gi}. Thus $p\in\Ass R$ if and only if $\dim_{R_p}M_p=-\inf M_p$. But $M$ being Cohen-Macaulay,
we know by ~\cite[Theorem 2.3 (d)]{sequence2} that this is further equivalent to $\dim R\slash p-\inf M_p=\dim_R M$. 
\end{proof}

\begin{Proposition}\label{cmdim} Let $(R,m)$ be a local ring admitting a dualizing complex.  If $R$ satisfies Serre's condition $S_2$ and $M\in D^f_b(R)$ is 
a G-Gorenstein complex, then $\dim_RM=\dim R-\sup M$. It follows that $\amp M=\cmd R$. In particular, if $R$ is Cohen-Macaulay, then  any G-Gorenstein complex is isomorphic to a module up to a suspension.
\end{Proposition}
\begin{proof}
Recall first that Serre's condition $S_2$ for $R$ implies that $\Ass R=\Assh R$ (see e.g.~\cite[Lemma 1.1]{aoyama}). This together with Proposition~\ref{ass w}
then shows that $\dim R\slash p=\dim R$ for any $p\in\Supp_RM$ with $\dim R\slash p-\inf M_p=\dim_R M$. Because $\Supp_RM=\Supp_RK_M$ by 
Corollary~\ref{supprtcm} a), we get $\dim_RK_M=\dim R$. Thereby the desired formula $\dim_RM=\dim R-\sup M$ follows from Corollary~\ref{propkkm1} a).
Since $\dim_RM=\depth R-\inf M$ by Proposition~\ref{d=gi}, this shows that $\amp M=\cmd R$. The last statement is now obvious.
\end{proof}

This gives immediately the following
\begin{Corollary}\label{s2cm}Let $(R,m)$ be a local ring admitting a dualizing complex and satisfying Serre's condition $S_2.$ If $R$ admits a G-Gorenstein module, then $R$ is Cohen-Macaulay.
\end{Corollary}

Let $R$ be a ring. Recall from~\cite[Definition 2.1]{Auscat} that a complex $C\in D^f_b(R)$ is said to be \textit{semi-dualizing} for $R$ if the homothety morphism $R\rightarrow{\bf R}\Hom_R(C,C)$ is an isomorphism in $D(R)$. It is natural to ask when a G-Gorenstein complex is semi-dualizing. 

\begin{Proposition}\label{semidualg-gor}Let $(R,m)$ be a local ring admitting a dualizing complex and let $M\in D^f_b(R)$ be a G-Gorenstein complex. Then the following statements are equivalent:
\begin{itemize}
\item[a)]$M$ is a semi-dualizing complex;
\item[b)]$M$ is a dualizing complex;
\item[c)]$K_M \cong R$;
\item[d)]$K_M$ is a semi-dualizing module.
\end{itemize}
\end{Proposition}
\begin{proof}
\smallskip
\noindent$a)\Rightarrow b)$: Because $M$ has finite Gorenstein injective dimension by Proposition~\ref{d=gi}, we know by~\cite[Proposition 8.4]{Auscat} that 
$M$ must be a dualizing complex.

\smallskip
\noindent$b)\Rightarrow c)$: By the uniqueness of the dualizing complex, we have $M\simeq \Sigma^{-t}D_R$ for some integer $t$. Then 
$M^{\dagger}\simeq \Sigma^{t}R$. Hence $\dim_RM=t$ by formula~$(\ref{sdagger})$ so that $K_M\cong\Ho_t( \textstyle  \Sigma^{t}R)\cong R$.

\smallskip
\noindent$c)\Rightarrow d)$: This is clear.

\smallskip
\noindent $d)\Rightarrow a)$: By Proposition~\ref{yz} $K_M\simeq\Sigma^{-\dim_RM}M^{\dagger}$. Using ``swap'' (see~\cite[A.4.22]{C.L.N}) we then obtain
\begin{align*}
{\bf R}\Hom_R(K_M,K_M)&\simeq{\bf R}\Hom_R(M^{\dagger},M^{\dagger})\\
&\simeq{\bf R}\Hom_R(M,M^{\dagger\dagger})\\&\stackrel{}\simeq {\bf R}\Hom_R(M,M),
\end{align*}
which implies the claim. 
\end{proof}

Let $(R,m)$ be a local ring  admitting a dualizing complex and let $M\in D^f_b(R)$ be a $G$-Gorenstein complex. We know by Proposition~\ref{semidualg-gor} 
that the biduality morphism $L\longrightarrow {\bf R}\Hom_R({\bf R}\Hom_R(L,M),M)$ cannot be an isomorphism for $L\in D^f_b(R)$ unless $M$ is dualizing. 
However, we will prove in Theorem~\ref{2} below that $\depth L$ is nevertheless preserved if $L$ has finite projective or injective dimension, and that this 
property characterizes Gorenstein complexes among the complexes of finite Gorenstein injective dimension. We first need a lemma.

\begin{Lemma}\label{1}Let $(R,m)$ be a local ring admitting a dualizing complex. If a complex $M\in D^f_b(R)$ has finite Gorenstein injective dimension, then
$$\width_R{\bf R}\Hom_R(L,M)=\depth_RL-\Gid_RM$$ for all complexes $L\in D_b(R)$ of finite projective or injective dimension. 
\end{Lemma}

\begin{proof}If $L$ has finite injective dimension, then~\cite[Theorem 6.3 (iii)]{ascent} and~\cite[Theorem 6.3]{c.h.f.functurial} immediately yield 
\begin{align*}
\width_R {\bf R}\Hom_R(L,M)&=\depth_R L +\width_R M-\depth R\\
&=\depth_RL-\Gid_RM.
\end{align*}

In the case $L$ has finite projective dimension, we know by~\cite[Theorem 4.7 (ii)]{ascent} that $\Gid {\bf R}\Hom_R(L,M)$ has finite Gorenstein injective dimension. 
So another application of~\cite[Theorem 6.3 (iii)]{ascent} gives
$$\width_R{\bf R}\Hom_R(D,{\bf R}\Hom_R(L,M))= \width_R{\bf R}\Hom_R(L,M) - \depth R,$$
since $\depth_RD=0$. On the other hand, by~\cite[Theorem 6.2 (ii)]{ascent} 
\begin{align*}
\width_R{\bf R}\Hom_R(L,&{\bf R}\Hom_R(D,M))\\
&=\depth_R L +\width_R{\bf R}\Hom_R(D,M) - \depth R\\
&=\depth_R L-\Gid_RM-\depth R,
\end{align*}
where the second inequality is by the already established case (take $L=D$). Since 
${\bf R}\Hom_R(D,{\bf R}\Hom_R(L,M))$ and ${\bf R}\Hom_R(L,{\bf R}\Hom_R(D,M))$ are
isomorphic by ``swap'' (see~\cite[A.4.22]{C.L.N}), we get
$\width_R{\bf R}\Hom_R(L,M)=\depth_R L-\Gid_RM$, as wanted.

\end{proof}

\begin{Theorem}\label{2}Let $(R,m)$ be a local ring admitting a dualizing complex and let $M\in D^f_b(R)$ be a complex of finite Gorenstein injective dimension. Then
the following statements are equivalent:
\begin{itemize}
\item[a)]$M$ is G-Gorenstein.
\item[b)]If $L\in D_b(R)$ has finite projective or injective dimension, then $$\depth_R {\bf R}\Hom_R({\bf R}\Hom_R(L,M),M) = \depth_RL;$$
\item[c)]$\depth_R {\bf R}\Hom_R(M,M) =  \depth R$;
\item[d)]$\depth_R {\bf R}\Hom_R({\bf R}\Hom_R(D_R,M),M) = 0$.
\end{itemize}
\end{Theorem}
\begin{proof}

In order to see the equivalence of a) and b) note that  
\begin{align*}\depth_R{\bf R}\Hom_R({\bf R}\Hom_R(L,M),M)&\stackrel{}=\width_R{\bf R}\Hom_R(L,M)+\depth_RM\\&\stackrel{}=\depth_RL-\Gid_RM+\depth_RM.
\end{align*}
The first equality comes from~\cite[Proposition 4.6]{iyengar} while the second one follows from Lemma~\ref{1}. Hence the equation \begin{equation*}\depth_R {\bf R}\Hom_R({\bf R}\Hom_R(L,M),M) = \depth_{R}L \end{equation*} is equivalent to $\depth_RM =\Gid_RM$. Noting that $\Gid_RM=\depth R-\inf M$ by~\cite[Theorem 6.3]{c.h.f.functurial}, the equivalence of $a)$ and $b)$ is then clear by Proposition~\ref{d=gi}. In fact, we observe that in order to a) hold, it is enough that b) holds from some $L$ of finite projective or injective dimension. In particular, we can take $L=R$ or $L=D_R$. So both $c)$ and $d)$ imply a). Since b) trivially implies both c) and d), we are done.
\end{proof}

\section{Two equivalences of categories}

\begin{Notation} Let $R$ be a ring. Let $t\in \mathbb Z$. We denote by $D_{t-GGor}(R)$ the full subcategory of $D^f_b(R)$ of G-Gorenstein complexes of dimension $t$. 
\end{Notation}

Our purpose is to show that both the equivalence of Yekutieli and Zhang considered in Proposition~\ref{yz} and Foxby equivalence restrict to an equivalence between the category 
$D_{t-GGor}(R)$ and the $G$-class $G(R)$.

\begin{Theorem}\label{asly1}Let $(R,m)$ be a local ring admitting a dualizing complex. For any $t\in \mathbb Z$, the equivalence of Proposition~\ref{yz}
induces an equivalence of categories
\[
\xymatrixcolsep{10pc}
\xymatrixrowsep{1pc}
\xymatrix{ 
D_{t-GGor}(R) 
\ar@<1ex>[r]^{K_{-}}& \G(R)^{\mathrm{opp}}\ar@<1ex>[l]^{\Sigma^{-t}(-)^{\dagger}}.\\
} 
\]
Furthermore, the following statements are equivalent for a complex $M\in D^f_b(R)$:
\begin {itemize}
\item[a)] $M\in D_{t-GGor}(R)$;
\item[b)] $M\simeq(\Sigma^{t}K_M)^\dagger$ and $K_M\in \G(R)$;
\item[c)] $M\simeq(\Sigma^{t}K)^\dagger$ for some $K\in \G(R)$.
\end{itemize}

\end{Theorem}
\begin{proof} 

The equivalence of a),b) and c) is clear as soon as we have established the claimed equivalence of categories. To do the latter, we need to show that
the restriction of the equivalence of Proposition~\ref{yz} makes sense. 

Suppose therefore that $M\in D_{t-GGor}(R)$. Of course $M\in D_{t-CM}(R)$. Now $\Ho^t_m(M)=\Hom_R(K_M,E_R(k))$. Since 
$\Ho^t_m(M)$ is Gorenstein injective, an application of~\cite[Theorem 6.4.2]{C.L.N} shows that $K_M$ is Gorenstein flat. So $K_M\in G(R)$.  

Conversely, take $K\in \G(R)$ and set $M= \Sigma^{-t}K^{\dagger}$. Then $M\in D_{t-CM}(R)$. By local duality $\Ho^t_m(M)\cong\Hom_R(K, E_R(k))$, so that $\Ho^t_m(M)$ is Gorenstein injective by~\cite[Theorem 6.4.2]{C.L.N}. Hence $M\in D_{t-GGor}(R)$ by Corollary~\ref{G-Gor for local} as wanted.

\end{proof}

Let us then consider the Foxby equivalence.

\begin{Theorem}\label{asly2}Let $(R,m)$ be a local ring admitting a dualizing complex. For any $t\in \mathbb Z$, Foxby equivalence induces an equivalence of categories
\[
\xymatrixcolsep{10pc}
\xymatrixrowsep{1pc}
\xymatrix{ 
D_{t-GGor}(R) 
\ar@<1ex>[r]^{H_{-t}({\bf R}\Hom_R(D_R,-))}& \G(R)\ar@<1ex>[l]^{\Sigma^{-t}D_R\otimes^L_R-}.\\
} 
\]
Furthermore, the following statements are equivalent for a complex $M\in D^f_b(R)$:
\begin {itemize}
\item[a)] $M\in D_{t-GGor}(R)$;
\item[b)] $M\simeq \Sigma^{-t}D_R\otimes ^{L}_RN$ for some $N\in \G(R)$;
\item[c)] ${\bf R}\Hom_R(\Sigma^{-t}D_R,M)\simeq N$ for some $N\in \G(R)$.
\end{itemize}
\end{Theorem}
\begin{proof}

Let us first check that the restriction of Foxby equivalence makes sense. Take $M\in D_{t-GGor}(R)$. Since $M$  by Proposition~\ref{d=gi} is of finite Gorenstein injective dimension,
we know that $M\in B^f(R)$. By Theorem~\ref{asly1} b) we have $M\simeq\Sigma^{-t}K_M^{\dagger}$, where $K_M\in \G(R)$. By~\cite[Lemma 2.7]{ff} and~\cite[Proposition 2.2.2]{C.L.N} we get 
$${\bf R}\Hom_R(\textstyle \Sigma^{-t}D_R,M) \simeq {\bf R}\Hom_R(D_R,K_M^{\dagger}) \simeq {\bf R}\Hom_R(K_M,R) \simeq\Hom_R(K_M,R).$$ 
This shows that $H_{-t}({\bf R}\Hom_R(D_R,M)) \in \G(R)$, as desired.

Conversely, let $N\in G(R)$. Set $M =\Sigma^{-t}D_R\otimes_R^LN$. By~\cite[Lemma 2.7]{ff} and~\cite[Proposition 2.2.2]{C.L.N}
$$M^{\dagger} \simeq\textstyle  \Sigma^t(D_R\otimes_R^LN)^{\dagger}\simeq \textstyle  \Sigma^t {\bf R}\Hom_R(N,R)\simeq\textstyle  \Sigma^t\Hom_R(N,R).$$
By formula~$(\ref{sdagger})$ $\dim M=t$. Since $\Hom_R(N,R)\in G(R)$, we have $M\in D_{t-GGor}(R)$ by Theorem~\ref{asly1} b).

The equivalence of a) and b) is now immediate. It is also clear that b) implies c). To see the converse, recall from~\cite[Theorem 3.3.2 (b)]{C.L.N} that ${\bf R}\Hom_R(\Sigma^{-t}D_R,M)\in A(R)$ implies $M\in B(R)$. By Foxby equivalence one then
has $$M\simeq \textstyle\Sigma^{-t}D_R\otimes_R^L {\bf R}\Hom_R(\textstyle\Sigma^{-t}D_R,M).$$

\end{proof}
\begin{Remark}\label{GandG}It follows from the above proof that the equivalences of Theorem~\ref{asly1} and Theorem~\ref{asly2} are compatible in the sense 
that the diagram mentioned in the Introduction is commutative up to a canonical isomorphisms. In fact, this compatibility can also be seen as a special case of~\cite[Lemma 2.7]{ff}.
\end{Remark}

\begin{Remark}It is easily checked that in the equivalences of Theorem~\ref{asly1} and Theorem~\ref{asly2} Gorenstein complexes correspond to finitely generated free modules.
In particular, this illustrates the fact that Gorenstein complexes form a proper subcategory of the category of $G$-Gorenstein complexes.
\end{Remark}

We will now look at the special case where $R$ is a Cohen-Macaulay ring admitting a canonical module $K_R$. Recall from Proposition~\ref{cmdim} that in this case every $G$-Gorenstein complex is isomorphic
to a module up to a suspension. Moreover, any $G$-Gorenstein module has dimension $\dim R$.

\begin{Notation} If $R$ is a ring, we denote by $GGor(R)$  the category of all G-Gorenstein modules.  
\end{Notation}

\begin{Corollary}\label{firstformodule}Let $(R,m)$ be a Cohen-Macaulay local ring admitting a canonical module $K_R.$ Then there exists a diagram
\[
\xymatrixcolsep{10pc}
\xymatrixrowsep{4pc}
\xymatrix{ 
GGor(R) 
\ar@<1ex>[r]^{K_{-}}
\ar[d]_{}
& \G(R)^{\mathrm{opp}}
\ar@<1ex>[l]^{K_{-}} 
\ar[d]^{\Hom_R(-,R)}
\\
GGor(R)
\ar@<1ex>[u]^{\id}
\ar@<1ex>[r]^{\Hom_R(K_R,-)}
& \G(R) 
\ar@<1ex>[l]^{K_R\otimes_R-} 
\ar@<1ex>[u]^{}
} 
\]
of equivalences of categories, 
where the horizontal arrows are quasi-inverses of each other. The diagram is commutative up to canonical isomorphisms. Furthermore,  if $M$ is a finitely generated $R$-module, then the following statements are equivalent:
\begin{itemize}

\item[a)]$M$ is a G-Gorenstein module;
\item[b)]$M$ is an equidimensional module satisfying  Serre's condition $S_2$ and $K_M\in G(R)$;
\item[c)]$M\cong K_R\otimes_R N$ for some $N\in G(R)$;
\item[d)]$\Hom_R(K_R,M)\in G(R)$.
\end{itemize} 
\end{Corollary}
\begin{proof}

Set $d=\dim R$. This is the diagram mentioned in Remark~\ref{GandG} in the case $t=d$. Indeed, $D_R\simeq\Sigma^dK_R$ by the Cohen-Macaulayness of $R$. If $N \in G(R)$, then by the Auslander-Bridger 
formula (see~\cite[Theorem 1.4.8]{C.L.N}) 
$N$ is a Cohen-Macaulay module of dimension $d$. So $\Sigma^{-d}N^{\dagger}\simeq K_N$. Moreover, using~\cite[Corollary 2.12]{c.h.f.functurial}, we now observe that 
$${\bf R}\Hom_R(D_R,M)\simeq \Hom_R(D_R,M) \simeq \Sigma^{-d}\Hom_R(K_R,M)$$
whereas by~\cite[Corollary 2.16]{c.h.f.functurial} 
$$\Sigma^{-d}D_R\otimes^L_RN\simeq \Sigma^{-d}D_R\otimes_R N \simeq K_R\otimes_R N$$
for all $N \in G(R)$.

To see the equivalence of a) and b), we can use the diagram. Indeed, if $M$ is $G$-Gorenstein, then $M\cong K_{K_M}$, where $K_M\in G(R)$. Note that the module 
 $K_{K_M}$ is equidimensional and satisfies $S_2$ by~\cite[Lemma 1.9, c) and e)]{schenzel}. Conversely, if $M$ is an equidimensional module satisfying  
$S_2$, then $M\cong K_{K_M}$ by~\cite[Proposition 1.1.4]{schenzel}. The equivalence of a), c) and d) follows directly from Theorem~\ref{asly2}. 

\end{proof}

As an application of Theorem~\ref{asly2} we will give one more criterium for a complex of finite Gorenstein injective dimension to be $G$-Gorenstein. For this,
we need the following well-known lemma, which we prove here for the convenience of the reader.

\begin{Lemma}\label{zero}Let $(R,m)$ be a local ring admitting a dualizing complex. If $M\in D^f_b(R)$, then
$${\bf R}\Hom_R(E_R(k),M)\simeq{\bf R}\Hom_R(D_R,M)\otimes_R\hat{R}.$$
\end{Lemma}

\begin{proof}By local duality, adjointness, \cite[Corollary 4.1.1 (ii)]{lipman2} and tensor evaluation (see~\cite[A.4.23]{C.L.N}), we get
\begin{align*}
{\bf R}\Hom_R(E_R(k),M)
&\simeq{\bf R}\Hom_R({\bf R}\Gamma_m(D_R),M)\\
&\simeq{\bf R}\Hom_R(D_R,{\bf R}\Hom_R({\bf R}\Gamma_m(R),M))\\
&\simeq{\bf R}\Hom_R(D_R,M\otimes_R\hat{R})\\
&\simeq{\bf R}\Hom_R(D_R,M)\otimes_R\hat{R}.
\end{align*} 
\end{proof}

We are now ready to prove the promised criterium. It is related to~\cite[Theorem 3.11 (v)]{sharp}.

\begin{Proposition}Let $(R,m)$ be a local ring admitting a dualizing complex. If $M\in D^f_b(R)$ has finite Gorenstein injective dimension, then the
following statements are equivalent:
\begin{itemize}
\item[a)]$M$ is G-Gorenstein of dimension $t$;
\item[b)] There exists a Gorenstein injective module $I$ and natural isomorphisms
\begin{equation*}{\bf R}\Hom_R(L,M)\simeq\textstyle  \Sigma^{-t}\Hom_R(L,I)
\end{equation*}
for all bounded complexes $L$ with $\Supp_RL=\left\{m\right\}$ consisting of either injective modules or projective modules.
\end{itemize}
\end{Proposition}
\begin{proof}
 \smallskip
\noindent$a)\Rightarrow b)$: Set $I=\Ho^t_m(M)$. We then know that $I$ is Gorenstein injective and ${\bf R}\Gamma_m(M)\simeq  \Sigma^{-t}I$. Now~\cite[Proposition 3.2.2]{lipman2}
 and~\cite[Corollary 2.12]{c.h.f.functurial} yield
$${\bf R}\Hom_R(L,M)\simeq{\bf R}\Hom_R(L,{\bf R}\Gamma_m(M))\simeq\textstyle  \Sigma^{-t}\Hom_R(L,I).$$

\smallskip
\noindent$b)\Rightarrow a)$: We want to use Theorem~\ref{asly2} c). Therefore we need to show that ${\bf R}\Hom_R(D_R,M)\simeq\Sigma^{-t}N$ for some $N\in G(R)$. 
We now have $${\bf R}\Hom_R(D_R,M)\otimes_R\hat{R}\simeq {\bf R}\Hom_R(E_R(k),M)$$ by Lemma~\ref{zero}. By assumption $${\bf R}\Hom_R(E_R(k),M)\simeq\textstyle  
\Sigma^{-t}\Hom_R(E_R(k),I).$$ It follows that ${\bf R}\Hom_R(D_R,M)\simeq\Sigma^{-t}N$ for some finitely generated $R$-module $N$. Now $\Hom_R(E_R(k),I)$ is Gorenstein 
flat by~\cite[Corollary 3.7 (c)]{ascent}. So $N\otimes_R\hat{R}$ is Gorenstein flat as an $R$-module. By~\cite[Lemma 2.6 (a)]{ascent} it is then Gorenstein flat also 
as an $\hat{R}$-module.  Therefore $N\in G(R)$ by~\cite[Theorem 8.7, (5)]{avbeijing}.

\end{proof}

\section{$G$-Gorenstein complexes as Gorenstein objects}

Let $\mathcal{C}$ be a class of objects in an abelian category $\mathcal{A}$. Consider an exact complex $$X:\quad \cdots\rightarrow X_{i+1}\stackrel{d_{i+1}}{\rightarrow} X_i \stackrel{d_{i}}{\rightarrow} X_{i-1}\stackrel{d_{i-1}}{\rightarrow}\cdots$$ in $\mathcal{A}$, where $X_i\in \mathcal C$ for all $i\in \mathbb Z$. Recall that $X$ is called \textit{$\mathcal{C}$-totally acyclic} if it is both $\Hom_{\mathcal{A}}(\mathcal{C},-)$-exact and $\Hom_{\mathcal{A}}(-,\mathcal{C})$-exact, i.e., the complexes $\Hom_{\mathcal{A}}(C,X)$ and $\Hom_{\mathcal{A}}(X,C)$ are exact in the category of abelian groups for any object $C$ in $\mathcal{C}$. 
A \textit{$\mathcal{C}$-Gorenstein object} is an object in $\mathcal{A}$ appearing as a kernel in a  $\mathcal{C}$-totally acyclic complex. In this section we want to show that in a certain sense $G$-Gorenstein complexes can be considered as Gorenstein objects in the nonabelian category  $D(R)$. 

We first need a suitable notion of exactness in a triangulated category. Our definition is a special case of the one Beligiannis gives in~\cite[Definition 4.7]{bel} (see also~\cite{asad}). In the definition $\Delta$ refers to the class of all exact triangles in a triangulated category $\mathcal D$ (see~\cite[Example 2.3]{bel}). We will always denote the suspension functor by $\Sigma$.

\begin{Definition} Let $\mathcal D$ be a triangulated category. A $\Delta$-exact complex in $\mathcal D$ is a diagram $$X\colon\quad   \cdots\rightarrow X_{i+1}\stackrel{d_{i+1}}{\rightarrow} X_i \stackrel{d_{i}}{\rightarrow} X_{i-1}\stackrel{d_{i-1}}{\rightarrow}\cdots  $$ of objects and morphisms in $\mathcal D$ such that there exists for all $i\in \mathbb Z$ an exact triangle $$M_{i}\stackrel{f_{i}}{\rightarrow} X_i\stackrel{g_{i}}{\rightarrow} M_{i-1}\rightarrow\textstyle  \Sigma M_{i}$$ where $d_i=f_{i-1}g_i$.
\end{Definition}

\begin{Remark}\label{differentialremark} By~\cite[Proposition 2.4 (a)]{asad}, one has $d_{i-1}d_i=0$ for all $i\in \mathbb Z$. Thus a diagram $X$ as above is indeed a complex.
\end{Remark}

The next two definitions are inspired by~\cite[Definition 3.2 and Definition 3.3]{asad}.

\begin{Definition}\label{exacttriangle} Let $\mathcal D$ be a triangulated category. Let $\mathcal{C}$ be a class of objects in $\mathcal D$. We say that an exact triangle $N\rightarrow M\rightarrow L\rightarrow\Sigma N$ in $\mathcal D$ is $\Hom_{\mathcal D}(\mathcal{C},-)$-exact if the induced sequence of abelian groups \begin{equation*}0\rightarrow\Hom_{\mathcal D}(C,N)\rightarrow\Hom_{\mathcal D}(C,M)\rightarrow\Hom_{\mathcal D}(C,L)\rightarrow 0\end{equation*} is exact for all $C$ in $\mathcal{C}$. The notion of a $\Hom_{\mathcal D}(-,\mathcal{C})$-exact triangle is defined analogously.
\end{Definition}

\begin{Definition} Let $\mathcal D$ be a triangulated category. Let $\mathcal{C}$ be a class of objects in $\mathcal D$. Consider a $\Delta$-exact complex 
$$X:\quad \cdots\rightarrow X_{i+1}\stackrel{d_{i+1}}{\rightarrow} X_i \stackrel{d_{i}}{\rightarrow} X_{i-1}\stackrel{d_{i-1}}{\rightarrow}\cdots  $$ 
in $\mathcal D$, where $X_i\in \mathcal C$ for all $i\in \mathbb Z$. We say that $X$ is totally $\mathcal{C}$-acyclic if all the associated exact triangles 
$$M_{i}\stackrel{f_{i}}{\rightarrow} X_i\stackrel{g_{i}}{\rightarrow} M_{i-1}\rightarrow\textstyle  \Sigma M_{i}$$ are both $\Hom_{\mathcal D}(\mathcal{C},-)$-exact and $\Hom_{\mathcal D}(-,\mathcal{C})$-exact.
\end{Definition}

\begin{Remark}\label{acyclicremark}If $X$ is a $\Delta$-exact complex in $\mathcal D$ whose associated triangles are $\Hom_{\mathcal D}(\mathcal{C},-)$-exact (resp.~$\Hom_{\mathcal D}(-,\mathcal{C})$-exact), then 
by pasting together the corresponding exact sequences of abelian groups, we see that the complex $\Hom_{\mathcal D}(C,X)$ (resp.~$\Hom_{\mathcal D}(X,C)$) is exact for all $C$ in $\mathcal{C}$. 
\end{Remark}

Let $R$ be a ring. Let $t\in \mathbb Z$. We aim next to investigate the relationship  between the notion of $\Delta$-exactness in $D^f_b(R)$ and the usual exactness in the abelian category $D_{t-CM}(R)$ of Cohen-Macaulay complexes of dimension $t$. For this we need some basic facts about t-structures.

\smallskip

Recall therefore from~\cite[D\'efinition 1.3.1]{A} that if $\mathcal D$ is a triangulated category, then a \textit{t-structure} on $\mathcal D$ is a pair $(C_{\geq 0},C_{\leq 0})$ of full subcategories of $\mathcal D$
satisfying the conditions:
\begin{itemize}
\item[1)] $\Sigma C_{\geq0}\subset C_{\geq0}$ and $\Sigma^{-1}C_{\leq 0}\subset C_{\leq 0}$;
\item[2)]If $M\in C_{\geq0}$ and $N\in \Sigma^{-1}C_{\leq 0}$ then $\Hom_{\mathcal D}(M,N)=0$;
\item[3)]If $M\in \mathcal D$, then there is an exact triangle $N\rightarrow M\rightarrow L\rightarrow\Sigma N$ with  $N\in C_{\geq0}$ and $L\in \Sigma^{-1}C_{\leq 0}$.
\end{itemize}
Set $C_{\geq n}=\Sigma^n C_{\geq 0}$ and $C_{\leq n}=\Sigma^n C_{\leq 0}$ for all $n\in \mathbb Z$. The \textit{heart} of the above t-structure is $\mathcal H:=C_{\geq 0}\cap C_{\leq 0}$. The heart 
is an abelian category. For the proof of this and the following fact, we refer to~\cite[Th\'eor\`eme 1.3.6]{A}.

\begin{Fact}\label{trianglefact}A sequence $$0\rightarrow X\stackrel{f}{\rightarrow} Y\stackrel{g}{\rightarrow} Z\rightarrow 0$$ in $\mathcal H$ is exact if and only if there exists
a morphism $h$ such that $$X\stackrel{f}{\rightarrow} Y\stackrel{g}{\rightarrow} Z\stackrel{h}{\rightarrow}\textstyle  \Sigma X$$ is an exact triangle in $\mathcal D$. 
\end{Fact}

\begin{Lemma}\label{123}Let $(R,m)$ be a local ring admitting a dualizing complex. For any $t\in \mathbb Z$, there is a t-structure on $D^f_b(R)$ whose heart is $D_{t-CM}(R)$. 
\end{Lemma}
\begin{proof} Let $(D_{\geq t} , D_{\leq t})$ be the so called standard t-structure on $D^f_b(R)$, where 
$$D_{\geq t}=\left\{X\in D^f_b(R)\mid \hbox{$H_i(X)=0$ for $i<t$}\right\}$$
and
$$D_{\leq t}=\left\{X\in D^f_b(R)\mid \hbox{$H_i(X)=0$ for $i>t$}\right\}.$$
By the dagger duality this gives raise to a t-structure $(D'_{\geq t} , D'_{\leq t})$, where
$$D'_{\geq t}=\left\{X\in D^f_b(R)\mid X^\dagger\in D_{\geq t}\right\}$$  
and 
$$D'_{\leq t}=\left\{X\in D^f_b(R)\mid X^\dagger\in D_{\leq t}\right\}.$$
Proposition~\ref{delta and cmd} combined with the local duality now implies that the heart of this t-structure is $D'_{\geq t} \cap D'_{\leq t} =D_{t-CM}(R)$.
\end{proof}

\begin{Proposition}\label{exactexact}Let $(R,m)$ be a local ring admitting a dualizing complex. Let $t\in \mathbb Z$, and set $D=\sum^{-t}D_R$. 
Consider a diagram $$X:\quad \cdots\rightarrow X_{i+1}\stackrel{d_{i+1}}{\rightarrow} X_i \stackrel{d_{i}}{\rightarrow} X_{i-1}\stackrel{d_{i-1}}{\rightarrow}\cdots $$
of objects and morphisms in $D_{t-CM}(R)$. Then $X$ is an exact complex in the abelian category $D_{t-CM}(R)$ if and only if it is a $\Delta$-exact complex in $D^f_b(R)$ with $\Hom_{D(R)}(-,D)$-exact associated triangles. Moreover, the associated triangles are 
$$M_{i}\stackrel{f_{i}}{\rightarrow} X_i\stackrel{g_{i}}{\rightarrow} M_{i-1}\rightarrow\textstyle\Sigma M_{i}$$ where $d_i=f_{i-1}g_i$ and $M_i$ denotes the kernel of $d_i$ in $D_{t-CM}(R)$ for every $i\in \mathbb Z$. 

\end{Proposition}

\begin{proof}
Suppose first that $X$ is exact in $D_{t-CM}(R)$. Let $M_i$ denote the kernel of $d_i$ in $D_{t-CM}(R)$ for every $i\in \mathbb Z$. By Fact~\ref{trianglefact} we get exact triangles $$M_{i}\stackrel{f_{i}}{\rightarrow} X_i\stackrel{g_{i}}{\rightarrow} M_{i-1}\rightarrow\textstyle  \Sigma M_{i},$$ where $d_i=f_{i-1}g_i$. So $X$ is $\Delta$-exact. Let us look at the long exact sequence of homology associated
to the functor $\Hom_{D(R)}(-, D)=H_t((-)^\dagger)$.  Since $M_i\in D_{t-CM}(R)$, we have $K_{M_i}^n=0$ for all $n\not=t$. We thus obtain the exact sequences $$0\rightarrow K_{M_{i-1}}\rightarrow K_{X_i}\rightarrow K_{M_i}\rightarrow 0$$
showing that the triangles are indeed $\Hom_{D(R)}(-,D)$-exact. 

Conversely, let $X$ be $\Delta$-exact complex in $D^f_b(R)$ with $\Hom_{D(R)}(-,D)$-exact associated exact triangles $$M_{i}\stackrel{f_{i}}{\rightarrow} X_i\stackrel{g_{i}}{\rightarrow} M_{i-1}\rightarrow\textstyle\Sigma M_{i}.$$  We will first show that every $M_i\in D_{t-CM}(R)$. Since $K_{X_i}^n=0$ for $n\not=t$, the long exact sequence of homology associated to the functor 
$\Hom_{D(R)}(-,D)$ gives for any $n\not=t$ an isomorphism $K^n_{M_{i}}\cong K^{n-1}_{M_{i-1}}$ and an exact sequence $$0\rightarrow K^{t+1}_{M_{i}}\rightarrow K^{t}_{M_{i-1}}\rightarrow K_{X_i}\rightarrow K^t_{M_i}\rightarrow K^{t-1}_{M_{i-1}}\rightarrow 0.$$ Our triangle now being $\Hom_{D(R)}(-,D)$-exact, we must have $K^{t+1}_{M_{i}}= K^{t-1}_{M_{i-1}}=0$. But then an easy induction shows that $K^n_{M_i}=0$ for all $n\neq t$. Thus $M_i\in D_{t-CM}(R)$. Fact~\ref{trianglefact} then shows that the sequences $$0\rightarrow M_i\stackrel{f_i}{\rightarrow}X_i\stackrel{g_i}{\rightarrow}M_{i-1}\rightarrow 0$$ are exact in $D_{t-CM}(R)$. Finally, we observe that now $\Ker d_{i}=\Ker g_{i}$ and $\Image d_{i+1}=\Image f_{i}$  implying that $X$ is an exact complex in $D_{t-CM}(R)$. 
\end{proof}

We can now prove the promised main result of this section.

\begin{Theorem}\label{dcomplete}Let $(R,m)$ be a local ring admitting a dualizing complex and let $M\in D^f_b(R)$. Let $t\in \mathbb Z$, and set $D=\sum^{-t}D_R$. Then $M$ is a G-Gorenstein complex of dimension $t$ if and only if there exists a $D$-totally acyclic complex $$ \cdots\rightarrow D^{\oplus n_{i+1}}\stackrel{d_{i+1}}{\rightarrow} D^{\oplus n_{i}} \stackrel{d_{i}}{\rightarrow} D^{\oplus n_{i-1}}\stackrel{d_{i-1}}{\rightarrow}\cdots$$ in $D^f_b(R)$ such that $M\simeq M_i$ where $M_i$ belongs to  some associated 
exact triangle $$M_{i}\stackrel{f_{i}}{\rightarrow} D^{\oplus n_{i}}\stackrel{g_{i}}{\rightarrow} M_{i-1}\rightarrow\textstyle  \Sigma M_{i}.$$ 

\end{Theorem}

\begin{proof}By Theorem~\ref{asly1} we know that $M\in D_{t-GGor}(R)$ if and only if $M\in D_{t-CM}$ and $K_M\in G(R)$. The latter means that $K_M$ appears as a cokernel in a totally acyclic complex of finitely generated free 
$R$-modules. In the equivalence of categories of Proposition~\ref{yz}
this complex corresponds to a $D$-totally acyclic complex
\begin{equation*}
\cdots\rightarrow D^{\oplus n_{i+1}}\stackrel{d_{i+1}}{\rightarrow} D^{\oplus n_{i}} \stackrel{d_{i}}{\rightarrow} D^{\oplus n_{i-1}}\stackrel{d_{i-1}}{\rightarrow}\cdots
\tag{$\ast$}
\end{equation*}
 in $D_{t-CM}(R)$. It follows that $M\in D_{t-GGor}(R)$ if and only if $M$ is isomorphic to a kernel in this complex.  

In light of Proposition~\ref{exactexact} and Remark~\ref{acyclicremark} it remains to show that if ($\ast$) is $D$-totally acyclic complex in $D_{t-CM}(R)$, then the corresponding
$\Delta$-exact complex in $D^f_b(R)$ has $\Hom_{D(R)}(D,-)$ exact associated triangles. Consider thus the triangles  $$M_{i}\stackrel{f_{i}}{\rightarrow} D^{\oplus n_{i}}\stackrel{g_{i}}{\rightarrow} M_{i-1}\rightarrow\textstyle  \Sigma M_{i},$$ where $d_i=f_{i-1}g_i$ and  $M_i$ is the kernel of $d_i$ in $D_{t-CM}(R)$ for all $i\in \mathbb Z$. Because $K_{M_i}\in G(R)$, the complex $M_i$ is
$G$-Gorenstein.  By Theorem ~\ref{asly2} c) we then have $H_i({\bf R}\Hom_R(D,M_i))=0$ if $i\not=0$. The long exact sequence of homology associated to the functor $\Hom_{D(R)}(D,-)=H_0({\bf R}\Hom_R(D,-))$ therefore yields the exact sequences $$0\rightarrow \Hom_{D(R)}(D,M_i)\rightarrow \Hom_{D(R)}(D,D^{\oplus n_{i}}) \rightarrow \Hom_{D(R)}(D,M_{i-1})\rightarrow 0$$
as needed.

\end{proof}

The following result is an immediate consequence of Theorem~\ref{dcomplete}.
\begin{Corollary}\label{K_R-exact}Let $(R,m)$ be a Cohen-Macaulay local ring admitting a canonical module $K_R.$ Let $M$ be an $R$-module. Then $M$ is a G-Gorenstein module if and only if $M$ is a kernel in an totally $K_R$-acyclic complex 
$$\quad\cdots\rightarrow K_R^{\oplus n_{i+1}}\stackrel{d_{i+1}}{\rightarrow}K_R^{\oplus n_i} \stackrel{d_i}{\rightarrow}K_R^{\oplus n_{i-1}}\rightarrow\cdots $$
of $R$-modules. 
\end{Corollary}

\end{document}